\documentclass[11pt]{amsart}

\usepackage[english]{babel}
\usepackage{amsmath}
\usepackage{amssymb}
\usepackage{amsthm}
\usepackage{latexsym}
\usepackage{mathtools}
\usepackage{thmtools}
\usepackage{amsfonts}
\usepackage{mathrsfs}
\usepackage{textcomp}
\usepackage[demo]{graphicx}
\usepackage{setspace}
\usepackage{nicefrac}
\usepackage{indentfirst}
\usepackage{enumerate}
\usepackage{wasysym}
\usepackage{upgreek}
\usepackage[pdfpagelabels,hyperindex=false]{hyperref}
\usepackage{paralist}
\usepackage{xcolor}
\hypersetup{
    colorlinks,
    linkcolor={red!50!black},
    citecolor={green!50!black},
    urlcolor={blue!80!black}
}%
\usepackage{etoolbox}
\usepackage{mathdots}
\usepackage{wrapfig}
\usepackage{floatflt}
\usepackage{tensor} 
\usepackage{parskip}
\usepackage{lscape}
\usepackage[enableskew]{youngtab}
\usepackage{pifont}

\usepackage{tikz}
\usetikzlibrary{arrows}
\usetikzlibrary{matrix}
\usetikzlibrary{positioning}
\usetikzlibrary{calc}

\usepackage[all,cmtip]{xy}
\usepackage[labelfont=bf, font={small}]{caption}[2005/07/16]

\usepackage{xcolor}
\definecolor{red}{rgb}{1,0,0}

\newcommand{\vvirg}{ , \dots , }
\newcommand{\ootimes}{ \otimes \cdots \otimes }

\newcommand{\calM}{\mathcal{M}}

\newcommand{\bbS}{\mathbb{S}}

\renewcommand{\phi}{\varphi}

\renewcommand{\theta}{\vartheta}

\renewcommand{\tilde}[1]{\widetilde{#1}}

\DeclareMathOperator{\Tr}{Tr}

\DeclareMathOperator{\End}{End}

\DeclareMathOperator{\sgn}{sgn}

\newcommand{\Mat}{\mathrm{Mat}}

\newcommand{\fillwidthof}[3][c]%
	{%
		\parbox
		{%
			\widthof{#2}%
		}%
		{%
			\ifx#1c%
				\centering#3%
			\else\ifx#1l%
				#3\hfill%
			\else\ifx#1r%
				\hfill#3%
			\fi\fi\fi%
		}%
	}%

\def\mylettrine#1#2 {\lettrine{#1}{#2}\space}

\newcommand{\partinto}[1][]{\smash{\mathord{\mathchoice{%
  \xymatrix@=0.4em@1{%
  \ar@{|-}[rr]_-*--{\scriptstyle #1}
  &*{\phantom{\scriptstyle{#1}}}&}
}{
  \xymatrix@=0.25em@1{%
  \ar@{|-}[rr]_-*--{\scriptstyle #1}
  &*{\phantom{\scriptstyle{#1}}}&}
}{
  \xymatrix@=0.2em@1{%
  \ar@{|-}[rr]_-*--{\scriptscriptstyle #1}
  &*{\phantom{\scriptscriptstyle{#1}}}&}
}{}}}}
\newcommand{\partintonosmash}[1][]{\mathord{\mathchoice{%
  \xymatrix@=0.4em@1{%
  \ar@{|-}[rr]_-*--{\scriptstyle #1}
  &*{\phantom{\scriptstyle{#1}}}&}
}{
  \xymatrix@=0.25em@1{%
  \ar@{|-}[rr]_-*--{\scriptstyle #1}
  &*{\phantom{\scriptstyle{#1}}}&}
}{
  \xymatrix@=0.2em@1{%
  \ar@{|-}[rr]_-*--{\scriptscriptstyle #1}
  &*{\phantom{\scriptscriptstyle{#1}}}&}
}{}}}
\newcommand{\partintostar}[1][]{\smash{\mathord{\mathchoice{%
  \xymatrix@=0.4em@1{%
  \ar@{|-}[rr]_-*--{\scriptstyle #1}^-*--{\scriptstyle \ast}
  &*{\phantom{\scriptstyle{#1}}}&}
}{
  \xymatrix@=0.25em@1{%
  \ar@{|-}[rr]_-*--{\scriptstyle #1}^-*--{\scriptstyle \ast}
  &*{\phantom{\scriptstyle{#1}}}&}
}{
  \xymatrix@=0.2em@1{%
  \ar@{|-}[rr]_-*--{\scriptscriptstyle #1}^-*--{\scriptstyle \ast}
  &*{\phantom{\scriptscriptstyle{#1}}}&}
}{}}}}
\newcommand{\partintostarnosmash}[1][]{\mathord{\mathchoice{%
  \xymatrix@=0.4em@1{%
  \ar@{|-}[rr]_-*--{\scriptstyle #1}^-*--{\scriptstyle \ast}
  &*{\phantom{\scriptstyle{#1}}}&}
}{
  \xymatrix@=0.25em@1{%
  \ar@{|-}[rr]_-*--{\scriptstyle #1}^-*--{\scriptstyle \ast}
  &*{\phantom{\scriptstyle{#1}}}&}
}{
  \xymatrix@=0.2em@1{%
  \ar@{|-}[rr]_-*--{\scriptscriptstyle #1}^-*--{\scriptstyle \ast}
  &*{\phantom{\scriptscriptstyle{#1}}}&}
}{}}}

\usepackage{bbold}
\usepackage[top=3cm,bottom=3cm,left=3cm,right=3cm]{geometry}
\usepackage{mathtools}
\usepackage{booktabs}
\usepackage{braket}
\usepackage{amssymb}
\usepackage{cleveref}

\DeclareMathOperator{\uMPS}{uMPS}
\DeclareMathOperator{\hMPS}{hMPS}

\DeclareMathOperator{\Cyc}{Cyc}
\DeclareMathOperator{\cyc}{cyc}

\def\CC{\mathbb{C}}
\def\NN{\mathbb{N}}

\def\SG{\mathfrak{S}}
\def\ot{\otimes}

\newtheorem{lemma}{Lemma}[section]

\newtheorem{claim}[lemma]{Claim}

\newtheorem{theorem}[lemma]{Theorem}
\newtheorem{corollary}[lemma]{Corollary}

\newtheorem{problem}[lemma]{Problem}
\newtheorem{observation}[lemma]{Observation}

\theoremstyle{definition}
\newtheorem{definition}[lemma]{Definition}
\newtheorem{remark}[lemma]{Remark}
\newtheorem{example}[lemma]{Example}
\title{Matrix product states, geometry, and invariant theory}
\author{Tim Seynnaeve}

\begin{document}
	\begin{abstract}	
		Matrix product states play an important role in quantum information theory to represent states of many-body systems. They can be seen as low-dimensional subvarieties of a high-dimensional tensor space. In these notes, we consider two variants: homogeneous matrix product states and uniform matrix product states. Studying the linear spans of these varieties leads to a natural connection with invariant theory of matrices. For homogeneous matrix product states, a classical result on polynomial identities of matrices leads to a formula for the dimension of the linear span, in the case of 2x2 matrices. %
	\end{abstract}
\maketitle

These notes are based partially on a talk given by the author at the University of Warsaw during the thematic semester ``AGATES: Algebraic Geometry with Applications to TEnsors and Secants", and partially on further research done during the semester. The author would like to thank Jarosław Buczyński, Weronika Buczyńska, Francesco Galuppi, and Joachim Jelisiejew for organizing this program, and Fulvio Gesmundo, Fatemeh Mohammadi, Cordian Riener, and Yang Qi for helpful discussions. This is still a preliminary version; an updated version will be uploaded over the course of 2023. Feel free contact me regarding any questions, comments, typos or errors. The final part contains ongoing research; if you are interested in working on the topic please contact me first.

\section{Invariant theory of matrices} \label{sec:InvThy}
In this section we give an introduction to the invariant theory of matrices. The exposition is based on the first chapters of \cite{DeConciniProcesi2017}. We work over the field of complex numbers.
\subsection{The invariant ring of a group action}
Let $V$ be a finite-dimensional representation of a group $G$. Then the ring 
\[
\CC[V] = S^{\bullet}(V^*)
\]
of polynomial functions on $V$ comes with an action of $G$, given by
\[
(g \cdot f)(v) = f(g^{-1}\cdot v)
\]
for $f \in \CC[V]$, $g \in G$, and $v \in V$. The \emph{invariant ring} is defined as 
\[
\CC[V]^G = \{f \in \CC[V] \mid g\cdot f = f \quad \forall g \in G\}.
\]
A central question in invariant theory is to, for a given $G$ and $V$, describe this ring, i.e.\
\begin{enumerate}
	\item Find generators for $\CC[V]^G$,
	\item Find the (polynomial) relations between these generators.
\end{enumerate}
\begin{example}
	Let $\SG_d$ be the symmetric group of order $d!$, acting on $V=\CC^d$ by permuting the coordinates. The invariant ring is the ring
	\[
	\CC[x_1,\ldots,x_d]^{\SG_d}
	\]
	of symmetric polynomials. It is known that this ring is generated by the elementary symmetric polynomials $e_1,\ldots,e_d$, where
	\[
	e_k := \sum_{i_1 < \ldots < i_k}{x_{i_1}\cdots x_{i_k}}.
	\]
	Furthermore, the $e_i$ are algebraically independent. In other words, the map \begin{align*}
	\CC[y_1,\ldots,y_d] &\xrightarrow{\cong} \CC[x_1,\ldots,x_d]^{\SG_d} \\
	y_i & \mapsto e_i
	\end{align*}
	is an isomorphism of $\CC$-algebras.
\end{example}

\subsection{The invariant ring of matrices} 
For the rest of this section, we will be concerned with one particular group action. We fix natural numbers $m$ and $n$, and take $G=GL_m(\CC)$ and $V=\Mat(m \times m, \CC)^n$ the space of $n$-tuples of $m \times m$ matrices. The action of $G$ on $V$ is given by simultaneous conjugation:
\[
g \cdot (A_1,\ldots,A_n) = (gA_1g^{-1}, \ldots, gA_ng^{-1}).
\]
We will write 
\[
\CC[V] = \CC[x^k_{ij} \mid 1\leq i,j \leq m; 1 \leq k \leq n] = \CC[X_1,\ldots,X_n],
\]
where $X_k = (x^k_{ij})_{ij}$ is an $m \times m$ matrix with generic entries. I.e.\ when we write $f(X_1,\ldots,X_n) \in \CC[X_1,\ldots,X_n]$, this means a polynomial in $nm^2$ variables. The invariant ring is given by
\[
\CC[X_1,\ldots,X_n]^{GL_m} = \{f \mid f(gX_1g^{-1}, \ldots, gX_ng^{-1}) = f(X_1,\ldots,X_n) \quad \forall g \in GL_m\}.
\]
Here the notation $f(gX_1g^{-1}, \ldots, gX_ng^{-1})$ means that we substitute the variable $x^k_{ij}$ with the $ij$-th entry of the matrix $gX_kg^{-1}$.

Note that $\CC[X_1,\ldots,X_n]$ admits an $\NN^n$ grading, by putting $\deg(x^k_{ij})$ equal to the $k$-th basis vector $e_k = (0,\ldots,1,\ldots,0) \in \NN^n$. It is not hard to see that a polynomial is invariant if and only if its graded parts are invariant; hence the invariant ring $\CC[X_1,\ldots,X_n]^{GL_m}$ inherits the aforementioned grading. 
The trick to describing this invariant ring is to start with its multilinear part
\[
\CC[X_1,\ldots,X_n]^{GL_m}_{(1,\ldots,1)}.
\]

\subsection{The multilinear part} \label{sec:multilinearPart}
The key ingredient is Schur-Weyl duality, a fundamental result in representation theory. We briefly state it here, and refer the reader to e.g.\ \cite[Chapter 6]{FultonHarris} for a more in-depth discussion.
\begin{theorem}
	Consider the vector space $(\CC^m)^{\ot n}$, equipped with the following actions of the general linear group $GL_m$ and the symmetric group $\SG_n$:
	\begin{align} \label{eq:leftGLaction}
		g \cdot (v_1 \ot \cdots \ot v_n) &= (g\cdot v_1) \ot \cdots \ot (g\cdot v_n), \\
		\sigma \cdot (v_1 \ot \cdots \ot v_n) &= v_{\sigma^{-1}(1)} \ot \cdots \ot v_{\sigma^{-1}(n)}. \label{eq:rightSGaction}
	\end{align}
	These actions commute, meaning that we get a ring homomorphism
	\begin{align} \label{eq:groupAlgebraToEnd}
		\CC[\SG_n] \to\End_{GL_m}((\CC^m)^{\ot n}).
	\end{align} 
	This homomorphism is surjective. If $n \leq m$, it is an isomorphism. If $n > m$, its kernel is given by the two-sided ideal generated by the antisymmetrizer
	\[
	c_{m+1} := \sum_{\sigma \in \SG_{m+1} \subseteq \SG_n} \sgn(\sigma)\cdot\sigma.
	\]
	We will denote this ideal by $I_{m,n}$.
\end{theorem}
The theorem below describes the multilinear part of $\CC[X_1,\ldots,X_n]^{GL_m}$ as a quotient of the symmetric group algebra $\CC[\SG_n]$. We will write elements in $\SG_n$ by cycle notation.
\begin{theorem} \label{thm:multilinearPart}
	There is a surjective linear map
	\begin{align*}
	\CC[\SG_n] &\twoheadrightarrow \CC[X_1,\ldots,X_n]^{GL_m}_{(1,\ldots,1)}\\
	(i_1\ldots i_k) \cdots (j_1 \ldots j_\ell) &\mapsto \Tr(X_{i_1}\cdots X_{i_k}) \cdots \Tr(X_{j_1}\cdots X_{j_\ell}).
	\end{align*}
	If $n \leq m$, it is an isomorphism. If $n > m$, its kernel is given by the two-sided ideal $I_{m,n}$ described above.
\end{theorem}
\begin{proof}[Proof sketch]
	Note that
	\[
	\CC[X_1,\ldots,X_n]_{(1,\ldots,1)} = S^{\bullet}\left(\End(\CC^m)^{\oplus n}\right)^*_{(1,\ldots,1)} \cong \left(\End(\CC^m)^{\ot n}\right)^* \cong \End\left((\CC^m)^{\ot n}\right),
	\]
	compatible with the $GL_m$-actions on both sides. But the ring 
	\[
	\End\left((\CC^m)^{\ot n}\right)^{GL_m} = \End_{GL_m}\left((\CC^m)^{\ot n}\right)
	\]
	is described by Schur-Weyl-duality.
\end{proof}
\begin{remark} \label{rmk:Imn}
	Recall the decomposition of $\CC[\SG_n]$ into isotypic components: 
	\[
	\CC[\SG_n] \cong \bigoplus_{\lambda \vdash n}{I_{\lambda}},
	\]
	compatible with both the left and right action of $\SG_n$. Then $I_{m,n}$ is the sum of all $I_{\lambda}$, where $\lambda$ is a Young diagram with at least $m+1$ rows. %
\end{remark}
\subsection{The other graded parts}
Fix $(d_1,\ldots, d_n)$ and write $d=\sum{d_i}$. Our goal is to describe the degree $(d_1,\ldots, d_n)$ part of our invariant ring. The main idea is that we can use \emph{polarization} to reduce this to the multilinear part.

More precisely: consider the ring $\CC[X_1^{(1)}, \ldots, X_1^{(d_1)}, \ldots, X_n^{(1)}, \ldots, X_n^{(d_n)}]$, which is just the ring $\CC[X_1,\ldots,X_{d}]$ where we relabeled the $d=\sum{d_i}$ matrices in a suggestive way. We equip it with the action of the group $\SG_{d_1}\times \cdots \times \SG_{d_n}$, where $\SG_{d_i}$ acts by permuting the (entries of the) matrices $X_i^{(1)}, \ldots, X_i^{(d_i)}$.
\begin{lemma} \label{lemma:polarization}
	We have an isomorphism of rings 
	\begin{align*}
		\CC[X_1^{(1)}, \ldots, X_1^{(d_1)}, \ldots, X_n^{(1)}, \ldots, X_n^{(d_n)}]_{(1,\ldots,1)}^{\SG_{d_1}\times \cdots \times \SG_{d_n}} &\cong \CC[X_1,\ldots,X_n]_{(d_1,\ldots,d_n)} \\
		f(X_1^{(1)},\ldots,X_1^{(d_1)},\ldots,X_n^{(1)},\ldots,X_n^{(d_n)}) &\mapsto f(X_1,\ldots,X_1,\ldots,X_n,\ldots,X_n).
	\end{align*}
	This isomorphism is compatible with the $GL_m$-action.
\end{lemma}
\begin{proof}
	The inverse map is given by polarization: given $g(X_1, \ldots, X_n) \in \CC[X_1,\ldots,X_n]_{(d_1,\ldots,d_n)}$, first substitute each $X_i$ for the sum $\sum_{j=1}^{d_i}{X_i^{(j)}}$, then take the multilinear part of the obtained polynomial. We leave it to the reader to check that, up to a nonzero scalar, this does define an inverse to the map above.
\end{proof}

Before we can state the main result of this section, we need some additional notation.
 Let $\SG_d$ be the group of permutations of the set
\[
\{(1,1),\ldots,(1,d_1),\ldots,(n,1),\ldots(n,d_n)\}.
\]
Let us now consider the subalgebra
\[
\CC[\SG_d]^{\SG_{d_1}\times \ldots \times \SG_{d_n}} \subseteq \CC[\SG_d]
\]
where the action of $\SG_{d_1}\times \ldots \times \SG_{d_n} \subset \SG_d$ on $\SG_d$ is given by conjugation. 
We will write elements of $\CC[\SG_d]^{\SG_{d_1}\times \ldots \times \SG_{d_n}}$ as (linear combinations of) ``pseudo-permutations" %
$\left(i_1\ldots i_k\right) \ldots \left(j_1 \ldots j_{\ell}\right)$, where the number $i$ appears $d_i$ times. By definition, $\left(i_1\ldots i_k\right) \ldots \left(j_1 \ldots j_{\ell}\right)$ is the image of any element of the form 
\[
\left((i_1,\cdot)\ldots (i_k,\cdot)\right) \ldots \left((j_1,\cdot) \ldots (j_{\ell},\cdot)\right)
\]
under the symmetrization map.
\begin{align*}
\CC[\SG_d] &\twoheadrightarrow \CC[\SG_d]^{\SG_{d_1}\times \ldots \times \SG_{d_n}} %
\end{align*}

\begin{theorem} \label{thm:otherGradedParts}
	There is a surjective linear map
	\begin{align} \label{eq:TraceRelationMap}
		\begin{split}
	\CC[\SG_d]^{\SG_{d_1}\times \ldots \times \SG_{d_n}} &\twoheadrightarrow \CC[X_1,\ldots,X_n]^{GL_m}_{(d_1,\ldots,d_n)}\\
	\left(i_1\ldots i_k\right) \cdots \left(j_1 \ldots j_{\ell}\right) &\mapsto \Tr(X_{i_1}\cdots X_{i_k}) \cdots \Tr(X_{j_1}\cdots X_{j_\ell}),
	\end{split}
	\end{align}
	whose kernel is given by 
	\[
	I_{m,d} \cap \CC[\SG_d]^{\SG_{d_1}\times \ldots \times \SG_{d_n}}.
	\]
	Here the action of $\SG_{d_1}\times \ldots \times \SG_{d_n} \subset \SG_d$ on $\SG_d$ is given by conjugation.
\end{theorem}
\begin{proof}[Proof sketch]
	Combine \Cref{lemma:polarization} and \Cref{thm:multilinearPart}, noting that in the map
	\[
	\CC[\SG_d] \twoheadrightarrow \CC[X_1^{(1)}, \ldots, X_1^{(d_1)}, \ldots, X_n^{(1)}, \ldots, X_n^{(d_n)}]_{(1,\ldots,1)},
	\]
	is compatible with the described actions of $\SG_{d_1}\times \ldots \times \SG_{d_n} \subset \SG_d$ (conjugation in $\SG_d$ on the left hand side; permuting matrices on the right hand side). 
\end{proof}
From \Cref{thm:otherGradedParts} one can deduce that our invariant ring $\CC[X_1,\ldots,X_n]^{GL_m}$ is (as a $\CC$-algebra) generated by polynomials of the form $\Tr(X_{i_1}\cdots X_{i_k})$. For this reason, it is also called the \emph{trace algebra}. In the future, we will use the notation $R_{m,n}:=\CC[X_1,\ldots,X_n]^{GL_m}$.
\subsection{Trace relations}
	Elements in the kernel of (\ref{eq:TraceRelationMap}) are known as \emph{trace relations}. They describe the polynomial relations that hold between our generators $\Tr(X_{i_1}\cdots X_{i_k})$.
	 Let us illustrate this by an example: take $m=2$ and $n=3$, and consider the degree $(1,1,1)$ part of the invariant ring $R_{2,3}=\CC[X_1,X_2,X_3]^{GL_2}$. By \Cref{thm:multilinearPart} this space is spanned by the 6 polynomials
	 \begin{align*}
	 &\Tr(X_1)\Tr(X_2)\Tr(X_3), &&\Tr(X_1X_2)\Tr(X_3), &&\Tr(X_1X_3)\Tr(X_2), \\ &\Tr(X_2X_3)\Tr(X_1), &&\Tr(X_1X_2X_3), &&\Tr(X_1X_3X_2).
	 \end{align*}
	 The ideal $I_{2,3}$ turns out to be just the linear span of $c_3=\sum_{\sigma \in \SG_3}{\sgn(\sigma)\cdot \sigma}$, which corresponds to the trace relation
	\begin{align*}
	&\Tr(X_1)\Tr(X_2)\Tr(X_3) - \Tr(X_1X_2)\Tr(X_3) - \Tr(X_1X_3)\Tr(X_2) \\ - &\Tr(X_2X_3)\Tr(X_1) + \Tr(X_1X_2X_3) + \Tr(X_1X_3X_2) = 0.
	\end{align*}
	The reader is invited to verify that the above expression really yields zero, for any $2 \times 2$-matrices $X_1,X_2,X_3$.
\section{Polynomial identities of matrices}
\begin{definition}
	Let $\CC\langle x_1,\ldots,x_n\rangle$ be the noncommutative polynomial ring in $n$ variables. An element $f \in \CC\langle x_1,\ldots,x_n\rangle$ is a \emph{polynomial identity for $m \times m$ matrices} if it vanishes when we substitute $x_1,\ldots,x_n$ by generic $m \times m$ matrices $X_1,\ldots,X_n$. The quotient of $\CC\langle x_1,\ldots,x_n\rangle$ by the ideal of polynomial identities is called the \emph{ring of generic $m \times m$ matrices}. We will denote this ring by $\calM_{m,n}$. %
\end{definition}

The ring of generic matrices is an important example of a \emph{Polynomial Identity (PI) ring}. Both PI-rings in general, and rings of generic matrices in particular, have been intensively studied since the 80's. We here present just a few concepts and results that will be relevant later. For more, the reader can consult \cite{DrenskyFormanek2004}.

\subsection{The Hilbert series} To quantify the number of polynomial identities, one can use the \emph{Hilbert series} of $\calM_{m,n}$. Recall that the Hilbert series of a multigraded vector space $R$ is given by
	\[
	H(R;t_1,\ldots,t_n) = \sum_{(d_1,\ldots,d_n) \in \NN^{n}}{\dim R_{(d_1,\ldots,d_n)}t_1^{d_1}\cdots t_n^{d_n}}.
	\]
The Hilbert series of $\CC\langle x_1,\ldots,x_n\rangle$ is given by 
\[
H(\CC\langle x_1,\ldots,x_n\rangle;t_1,\ldots,t_n) = \frac{1}{1-t_1-\cdots-t_n}.
\]
If we also know the Hilbert series of $\calM_{m,n}$, we can compute the dimension of the space of polynomial identities of multidegree $(d_1,\ldots,d_n)$ as the coefficient of $t_1^{d_1}\cdots t_n^{d_n}$ in the difference
\[
\frac{1}{1-t_1-\cdots-t_n} - H(\calM_{m,n};t_1,\ldots,t_n).
\]

Similar to the trace relations in the previous sections, one can use polarization to show it is sufficient to restrict to the multilinear part. Note that $\calM_{m,n}$ is a representation of $GL_n$, as it is the quotient of the representation
\[
\CC\langle x_1,\ldots,x_n\rangle \cong \bigoplus_{d \in \NN}{(\CC^n)^{\ot d}}
\]
 by a subrepresentation. Let us denote the degree $(1,\ldots,1)$ part of ${\calM_{m,d}}$ by $P_{m,d}$. This is a quotient of 
 \[
 \CC\langle x_1,\ldots,x_d\rangle_{(1,\ldots,1)}, %
 \]
which is equipped with a left $\SG_d$-action via
\[
\sigma \cdot (x_{i_1} \cdots x_{i_d}) = x_{\sigma(i_1)} \cdots x_{\sigma(i_d)}. 
\]
Then we have the following theorem, which was independently proved by Berele \cite[Theorem 2.7]{Berele1982} and Drensky (\cite[Remark 1.5]{Drensky1981Representations} and \cite[Lemma 2.3]{Drensky1984Codimensions}). The present formulation is based on \cite[Theorem 2.3.4]{DrenskyFormanek2004}. 
\begin{theorem} \label{thm:SGandGLdecompositionsAgree}
	Suppose the decomposition of $P_{m,d}$ into irreducible $\SG_d$-representations is given by
	\[
	P_{m,d} = \bigoplus_{\lambda \vdash d}V_{\lambda}^{\oplus a_\lambda}.
	\]
	Then the decomposition of $\calM_{m,n}$ into irreducible $GL_n$-representations is given by
	\[
	\calM_{m,n} = \bigoplus_{d \in \NN}\bigoplus_{\substack{\lambda \vdash d \\ \operatorname{len}(\lambda)\leq n}}{\bbS^{\lambda}(\CC^n)^{\oplus a_\lambda}}.
	\]
	Here for $\lambda \vdash d$ a partition, $V_{\lambda}$ denotes the associated Specht module, and $\bbS^\lambda$ the associated Schur functor.
\end{theorem}
\begin{corollary} \label{cor:HilbertSeries}
	The Hilbert series of $\calM_{m,n}$ is given by
	\[
	H(\calM_{m,n};t_1, \ldots, t_n) = \sum_{d \in \NN}\sum_{\substack{\lambda \vdash d \\ \operatorname{len}(\lambda)\leq n}}{a_{\lambda}s_\lambda(t_1, \ldots, t_n)},
	\]
	where $a_{\lambda}$ are the coefficients from \Cref{thm:SGandGLdecompositionsAgree}, and $s_\lambda$ denotes the Schur polynomial.
\end{corollary}

In general, it is a hard task to compute the multiplicities $a_\lambda$. However, in the case of $2 \times 2$ matrices ($m=2$) a formula is known \cite{Drensky1984Codimensions,Formanek1984Invariants}:
\begin{theorem} \label{thm:decompositionFor2x2}
For $m=2$, the coefficients $a_{\lambda}$ above are $0$ whenever $\operatorname{len}(\lambda) > 4$. If $\operatorname{len}(\lambda) \leq 4$, then $a_{\lambda}$ is given by the table below.
\begin{center}
\begin{tabular}{c|c}
	$\lambda$ & $a_\lambda$ \\
	\hline
	$(\lambda_1,0,0,0)$ & $1$ \\
	$(\lambda_1,1,0,0)$ & $\lambda_1$ \\
	$(\lambda_1,\lambda_2,0,0)$ with $\lambda_2>1$ & $(\lambda_1-\lambda_2+1)\lambda_2$ \\
	$(\lambda_1,1,1,0)$ & $2\lambda_1-1$ \\
	$(\lambda_1,1,1,1)$ & $\lambda_1-1$ \\
	$(\lambda_1,\lambda_2,\lambda_3,\lambda_4)$ not in the cases above & $(\lambda_1-\lambda_2+1)(\lambda_2-\lambda_3+1)(\lambda_3-\lambda_4+1)$ \\
\end{tabular}	
\end{center}
\end{theorem}
Combining \Cref{thm:decompositionFor2x2} with \Cref{cor:HilbertSeries} yields the Hilbert series of $\calM_{2,n}$, for any $n$. For fixed small $n$, the obtained expression can be simplified significantly. We present it here in the case $n=2$. 
\begin{theorem}[{\cite{Formanek1984Invariants}}]
	The Hilbert series of the ring two generic $2 \times 2$-matrices is given by
	\[
	H(\calM_{2,2};t_1,t_2) = \frac{1-t_1-t_2+t_1t_2+t_1^2t_2+t_1t_2^2-t_1^2t_2^2}{(1-t_1)^2(1-t_2)^2(1-t_2t_2)}.
	\]
\end{theorem}

\subsection{Relation to the trace algebra} 
\begin{definition}
	By $R_{m,n}^{\cyc}\subseteq R_{m,n}$ we mean the linear subspace spanned by elements of the form $\Tr(g(X_1,\ldots,X_{n}))$, where $g(x_1,\ldots,x_{n}) \in \CC \langle x_1,\ldots,x_{n} \rangle$. 
	Furthermore, let $\CC[\SG_d]_{\cyc} \subseteq \CC[\SG_d]$ be the subspace spanned by permutations consisting of a single cycle of length $d$.
\end{definition}
Note that \eqref{eq:TraceRelationMap} restricts to a surjective map
\begin{align} \label{eq:TraceRelationMapCyclic}
	\begin{split}
		\CC[\SG_d]_{\cyc}^{\SG_{d_1}\times \ldots \times \SG_{d_n}} &\twoheadrightarrow R^{\cyc}_{m,n;(d_1,\ldots,d_n)}\\
		\left(i_1\ldots i_d\right)&\mapsto \Tr(X_{i_1}\cdots X_{i_d})
	\end{split}
\end{align}
whose kernel is given by
\begin{equation}\label{eq:kernel}
I_{m,d} \cap \CC[\SG_d]_{\cyc}^{\SG_{d_1}\times \ldots \times \SG_{d_n}}.
\end{equation}
\begin{example}
	$R_{2,3;(1,1,1)}^{\cyc}$ is the linear span of $\Tr(X_1X_2X_3)$ and $\Tr(X_1X_3X_2)$, and $\CC[\SG_d]_{\cyc}$ is the linear span of $(123)$ and $(132)$. In this case, the map \eqref{eq:TraceRelationMapCyclic} is an isomorphism.
\end{example}

Consider the linear map 
\begin{align*}
	T: \CC\langle x_1, \ldots, x_n \rangle & \to \CC[X_1,\ldots,X_n,X_{n+1}] \\
	f(x_1,\ldots,x_n) & \mapsto \Tr(f(X_1,\ldots,X_n)\cdot X_{n+1})
\end{align*}
Note that this actually lands in the invariant ring $\CC[X_1,\ldots,X_n,X_{n+1}]^{GL_m} = R_{m,n+1}$. Moreover, one verifies that 
\begin{itemize}
	\item the kernel of $T$ is the ideal of polynomial identities.
	\item the image of $T$ is the subspace of $R_{m,n+1}^{\cyc}$ consisting of all trace polynomials that have degree $1$ the variable $X_{n+1}$. 
\end{itemize}

In other words, $T$ induces isomorphisms 
\begin{equation} \label{eq:PIAndTraceAlg}
	\calM_{m,n;(d_1,\ldots,d_n)} \cong R_{m,n+1;(d_1,\ldots,d_n,1)}^{\cyc}.
\end{equation}

\begin{remark}
Combining \eqref{eq:PIAndTraceAlg} and \eqref{eq:TraceRelationMapCyclic}, we get a surjective linear map
\begin{align*}
	\CC[\SG_{d+1}]^{\SG_{d_1}\times \ldots \times \SG_{d_n} \times \SG_1}_{\text{cyc}} & \twoheadrightarrow \calM_{m,d;(d_1,\ldots,d_n)} \\
	(i_1, \ldots, i_n, d+1) &\mapsto X_{i_1} \cdots X_{i_n}
\end{align*} 
whose kernel is given by
\begin{equation} \label{eq:linearEquationshMPS}
	I_{m,d+1} \cap \CC[\SG_{d+1}]^{\SG_{d_1}\times \ldots \times \SG_{d_n} \times \SG_1}_{\text{cyc}}.
\end{equation}

In particular, the space of multilinear polynomial identities in $d$ matrices can be identified with the intersection 
\[
I_{m,d+1} \cap \CC[\SG_{d+1}]_{\cyc}.
\]

In other words, we have
\begin{equation} \label{eq:Pmdcyc}
P_{m,d} \cong \frac{\CC[\SG_{d+1}]_{\cyc}}{I_{m,d+1} \cap \CC[\SG_{d+1}]_{\cyc}}.
\end{equation}

Recall that $P_{m,d}$ came equipped with an $\SG_d$-action. Under the isomorphism \eqref{eq:Pmdcyc}, this action corresponds to the action of $\SG_d \subset \SG_{d+1}$ on $\CC[\SG_{d+1}]$ by conjugation.
\end{remark}

\section{Matrix product states}

Matrix product states arise in the context of quantum many-body systems, where they are used to model ground states of certain systems. We will be concerned with two variants which model translation-invariant systems. For the purpose of these notes, we think about them as low-dimensional subvarieties of a high-dimensional tensor space. For more on matrix product states and their geometry, we refer to \cite{PerezGarciaVerstraeteMPS,haegeman2014geometry,critchMorton,TS:uMPS}.

\begin{definition}
	The variety $\uMPS(m,n,d)$ of \emph{uniform matrix product states} is the closed image of the map
	\begin{align}\label{eq:umpsMap}
		\begin{split}
			 (\CC^{m\times m})^n &\to (\CC^n)^{\otimes d} \\
			(A_1 \vvirg A_{n})&\mapsto \sum_{1 \leq i_1\vvirg i_d \leq n} \Tr(A_{i_1}\cdots A_{i_d}) \ e_{i_1}\ootimes e_{i_d}.
		\end{split}
	\end{align}
	The variety $\hMPS(m,n,d)$ of \emph{homogeneous matrix product} states is the closed image of the map
	\begin{align}\label{eq:hmpsMap}
		\begin{split}
			  (\CC^{m\times m})^{n+1} &\to (\CC^n)^{\otimes d} \\
			(X,A_1 \vvirg A_{n})&\mapsto \sum_{1 \leq i_1\vvirg i_d \leq n} \Tr(X A_{i_1}\cdots A_{i_d}) \ e_{i_1}\ootimes e_{i_d}.
		\end{split}
	\end{align}
\end{definition}

The goal of this section is to study the linear spaces spanned by these varieties. For previous work on this topic, see \cite{NV} ($\hMPS$) and \cite{TS:uMPSlinear} ($\uMPS$).

\subsection{Homogeneous matrix product states}
If we write $e_I := e_{i_1} \ot \cdots \ot e_{i_d}$ for a word $I=(i_1, \ldots, i_d)$, then
the space $(\CC^n)^{\otimes d}$ has a basis $\{e_I\}_{I \in [n]^d}$. We will denote the dual basis by $\{x_I\}_{I \in [n]^d}$. The following observation links homogeneous matrix product states to polynomial identities of matrices:
\begin{observation}
	An equation $\sum_I{\lambda_I x_I}$ vanishes on $\hMPS(m,n,d)$ if and only if the identity $\sum_I{\lambda_I\Tr(X A_{i_1}\cdots A_{i_d})}=0$ holds for all matrices $X,A_i$, if and only if $\sum_I{\lambda_I A_{i_1}\cdots A_{i_d}}=0$ is a polynomial identity of matrices. 
\end{observation}

\begin{corollary}
	The linear span of $\hMPS(m,n,d)$ has dimension equal to $\dim \calM_{m,n;d}$, the degree $d$ part of the ring of $n$ generic $m\times m$-matrices. In other words, we have
	\[
 \sum_{d\in \NN}{\dim \langle \hMPS(m,n,d) \rangle t^d} = H(\calM_{m,n};t,\ldots,t).
	\]
\end{corollary}

\begin{example} For $m=n=2$, we have
	\begin{align*}
		\sum_{d\in \NN}{\dim \langle \hMPS(2,2,d) \rangle t^d} =& H(\calM_{2,2};t,t)\\ 
		=& \frac{1-2t+t^2+2t^3-t^4}{(1-t)^4(1-t^2)} \\
		=& 1 + 2t + 4t^2 + 8t^3 + 16t^4 + 30t^5 + 53t^6 + 88t^7 + 139t^8 + 210t^9 \\
		 & + 306t^{10} + 432t^{11} + 594t^{12} + 798t^{13} + 1051t^{14} + 1360t^{15} + \cdots %
	\end{align*}
	Equivalently, we have the following explicit formula:
	\[
	\dim \langle \hMPS(2,2,d) \rangle = \begin{cases}
		\frac{1}{48}(d^4+4d^3+2d^2+44d+48) & \text{for } d \text{ even,} \\
		\frac{1}{48}(d^4+4d^3+2d^2+44d+45) & \text{for } d \text{ odd.}
	\end{cases}
	\]
	Compare with the top row of Table 1 in \cite{NV}. For $m=2$ and larger $n$, it is possible to write down a similar formula using \Cref{thm:decompositionFor2x2}. However, larger $m$ are much more difficult: even for $m=3$ and $n=2$ no exact formula is known to the best of my awareness.
\end{example}

\subsection{Uniform matrix product states}
	Since traces of matrices are invariant under cyclic permutations, the variety $\uMPS(m,n,d)$ is contained in the linear subspace $\Cyc^d(\CC^n) \subseteq (\CC^n)^{\ot d}$ of tensors that are invariant under cyclic permutations.  
	Consider the equivalence relation on $[n]^d$ given by identifying words with their cyclic permutations. The equivalence classes are known as \emph{necklaces} of length $d$ on the alphabet $[n]$; we will denote the set of such necklaces by $N_d([n])$.
	For a necklace $N \in N_d([n])$, we define $e_N := \sum_{(i_1,\ldots,i_d) \in N}{e_{i_1}\ot \cdots \ot e_{i_d}}$. Then $\{e_N\}_{N \in N_d([n])}$ is a basis of $\Cyc^d(\CC^n)$.
	We will denote the dual basis by $\{y_N\}_{N \in N_d([n])}$.
	
	As before, we see that linear equations $\sum_N{\lambda_N y_N}$ vanishing on $\uMPS(m,n,d)$ correspond to trace relations $\sum_N{\lambda_N \Tr(A_{i_1}\cdots A_{i_d})}=0$ , where $(i_1,\ldots,i_d)$ is a representative of the necklace $N$. Hence we find:
	\begin{observation} \label{obs:uMPSLinearSpan}
		The linear span of $\uMPS(m,n,d)$ has dimension equal to $\dim R^{\cyc}_{m,n;d}$. In other words:
		\[
		\sum_{d\in \NN}{\dim \langle \uMPS(m,n,d) \rangle t^d} = H(R^{\cyc}_{m,n};t,\ldots,t).
		\]
	\end{observation}
	In order to determine the Hilbert series of $R^{\cyc}_{m,n}$, we can try the same approach as for $\calM_{m,n}$. In particular, I claim that the analogue of \Cref{thm:SGandGLdecompositionsAgree} holds. Let us write
	\[
	\tilde{P}_{m,d}:=R^{\cyc}_{m,d;(1,\ldots,1)}.
	\]
	\begin{claim} \label{claim:SGandGLdecompositionsAgreeBis}
		Suppose the decomposition of $\tilde{P}_{m,d}$ into irreducible $\SG_d$-representations is given by
		\[
		\tilde{P}_{m,d} = \bigoplus_{\lambda \vdash d}V_{\lambda}^{\oplus b_\lambda}.
		\]
		Then the decomposition of $R^{\cyc}_{m,n}$ into irreducible $GL_n$-representations is given by
		\[
		R^{\cyc}_{m,n} = \bigoplus_{d \in \NN}\bigoplus_{\substack{\lambda \vdash d \\ \operatorname{len}(\lambda)\leq n}}{\bbS^{\lambda}(\CC^n)^{\oplus b_\lambda}}.
		\]
		Hence, the Hilbert series of $R^{\cyc}_{m,n}$ is given by
		\[
		H(R^{\cyc}_{m,n};t_1, \ldots, t_n) = \sum_{d \in \NN}\sum_{\substack{\lambda \vdash d \\ \operatorname{len}(\lambda)\leq n}}{b_{\lambda}s_\lambda(t_1, \ldots, t_n)}.
		\]
	\end{claim}
\begin{proof}
	Will be done in the next version of these notes. The idea is to use polarization.
\end{proof}
	We have 
	\[
	\tilde{P}_{m,d} \cong \frac{\CC[\SG_{d}]_{\cyc}}{I_{m,d} \cap \CC[\SG_{d}]_{\cyc}} \cong P_{m,d-1}.
	\]
	In particular, for $m=2$, \Cref{thm:decompositionFor2x2} gives a decomposition of $\tilde{P}_{m,d}$ as a representation of the subgroup $\SG_{d-1} \subset \SG_{d}$. But this is not enough information to deduce the decomposition as an $\SG_{d}$-representation:
	\begin{problem}
	Write down a closed formula for the decomposition of the $\SG_d$-representation
	\[
	\tilde{P}_{2,d} = R^{\cyc}_{2,d;(1,\ldots,1)} = \frac{\CC[\SG_{d}]_{\cyc}}{I_{2,d} \cap \CC[\SG_{d}]_{\cyc}}.
	\] 
	\end{problem}
	Via \Cref{obs:uMPSLinearSpan} and \Cref{claim:SGandGLdecompositionsAgreeBis}, such a formula would give a formula for the dimension of $\langle \uMPS(2,n,d) \rangle$. In particular, it would solve Conjectures 3.7 and 3.8 in \cite{TS:uMPSlinear}.
	\section*{Acknowledgement}
	This work is partially supported by Research foundation -- Flanders (FWO) -- Grant Number 1219723N, and partially supported by the Thematic Research Programme ``Tensors: geometry, complexity and quantum entanglement", University of Warsaw, Excellence Initiative – Research University and the Simons Foundation Award No. 663281 granted to the Institute of Mathematics of the Polish Academy of Sciences for the years 2021-2023. 
	\bibliographystyle{alpha} 
	\bibliography{MPSInvThy}{}
\end{document}